\documentclass[12pt,reqno]{amsart}

\usepackage{amssymb,dirtytalk,comment,tikz-cd,tikz,color,cite,enumerate,setspace}
\usepackage[unicode=true]
 {hyperref}
\hypersetup{
colorlinks=true,
urlcolor=black,
citecolor=blue,
linkcolor=blue,
}
\usepackage{xltabular}
\usepackage{pdflscape}

\newtheorem{theorem}{Theorem}[section]
\newtheorem{lem}[theorem]{Lemma}
\newtheorem{prop}[theorem]{Proposition}

\theoremstyle{definition}
\newtheorem{definition}[theorem]{Definition}

\newtheorem*{notation*}{Notation}
\theoremstyle{remark}
\newtheorem{remark}[theorem]{Remark}

\numberwithin{equation}{section}

\DeclareMathOperator{\Aut}{Aut}

\newcolumntype{L}[1]{>{\raggedright\arraybackslash}p{#1}}
\newcolumntype{C}[1]{>{\centering\arraybackslash}p{#1}}
\newcolumntype{R}[1]{>{\raggedleft\arraybackslash}p{#1}}

\allowdisplaybreaks
\begin{document}

\title[The automorphism group of a skew brace]{On the triviality and non-triviality of the automorphism group of a skew brace}

\author{Cindy (Sin Yi) Tsang}
\address{Department of Mathematics, Ochanomizu University, 2-1-1 Otsuka, Bunkyo-ku, Tokyo, Japan}
\email{tsang.sin.yi@ocha.ac.jp}
\urladdr{http://sites.google.com/site/cindysinyitsang/}



\begin{abstract}It is a simple fact that a group has a trivial automorphism group if and only if it is of order $1$ or $2$. We prove that the same holds for certain families of skew braces, and given any odd prime $p$, we construct a skew brace of order $2p^3$ that has a trivial automorphism group. \end{abstract}

\maketitle


\vspace{-5mm}

\section{Introduction}

For any group $G$, not necessarily finite, it is well-known that
\[
\Aut(G) \mbox{ is trivial} \iff G\mbox{ has order $1$ or $2$}.
\]
This is an immediate consequence of the following observations:
\begin{enumerate}[$\bullet$]
\item If $G$ is non-abelian, then $\mathrm{Inn}(G)$ is non-trivial.
\item If $G$ is abelian and $\exp(G) \geq 3$, then the inversion map $g\mapsto g^{-1}$ is a non-trivial element of $\Aut(G)$.
\item If $\exp(G)=2$ and $|G| \geq 3$, then $G$ may be treated as a vector space over $\mathbb{F}_2$ with $\dim_{\mathbb{F}_2}(G) \geq 2$, and any non-trivial permutation of some fixed set of basis vectors induces a non-trivial element of $\Aut(G)$.
\item If $|G|\leq 2$, then obviously $\Aut(G)$ is trivial.
\end{enumerate}
The purpose of this paper is to study the triviality or non-triviality of the automorphism group of a skew brace.

A \textit{skew brace} is any set $A = (A,\cdot,\circ)$ equipped with two group operations such that the so-called brace relation
\[ a \circ (b\cdot c) = (a\circ b)\cdot a^{-1}\cdot (a\circ c)\]
holds for all $a,b,c\in A$. For any $a\in A$, we use $a^{-1}$ and $\overline{a}$, respectively, to denote its inverses in the \textit{additive} group $(A,\cdot)$ and the \textit{multiplicative} group $(A,\circ)$. We also use $1$ to denote the common identity element of the two groups. The automorphism group of $A$ is defined to be
\[ \Aut(A) = \Aut(A,\cdot)\cap \Aut(A,\circ).\]
If $|A|=1$ or $2$, then clearly $\Aut(A)$ is trivial. However, the converse is false. Using the \texttt{YangBaxter} package \cite{YangBaxter} in \texttt{GAP} \cite{GAP}, we found that
\[ \texttt{SmallSkewbrace(24,855)},\quad \texttt{SmallSkewbrace(54,i)},\]
where $157\leq i \leq 164$ and $721\leq i\leq 722$, have a trivial automorphism group. See Section \ref{sec:24} for a construction of the skew brace of order $24$.

A complete classification of the skew braces with a trivial automorphism group seems to be a very challenging problem. The objective of this paper is to make some progress by showing that:
\begin{enumerate}[(i)]
\item For certain families of skew braces, it is true that having order at least $3$ implies that the automorphism group is non-trivial.
\item For any odd prime $p$, there exists a skew brace of order $2p^3$ whose automorphism group is trivial.
\end{enumerate}
Recall that a skew brace $A = (A,\cdot,\circ)$ is said to be \textit{two-sided} if
\[ (b\cdot c)\circ a = (b\circ a) \cdot a^{-1}\cdot (c\circ a)\]
holds for all $a,b,c\in A$, and is said to be a \textit{bi-skew brace} if 
\begin{equation}\label{eqn:bi} a\cdot (b\circ c) = (a\cdot b)\circ \overline{a} \circ (a\cdot c)\end{equation}
holds for all $a,b,c\in A$. Here are our main results:

\begin{theorem}\label{thm}
    Let $A= (A,\cdot,\circ)$ be a skew brace with $|A|\geq 3$ such that 
    \begin{enumerate}[$(1)$]
    \item $A$ is a two-sided skew brace with $(A,\circ)$ non-abelian; or
     \item $A$ is a bi-skew brace; or
    \item $A$ is finite and $(A,\cdot), (A,\circ)$ are both nilpotent.
    \end{enumerate}
    Then $\Aut(A)$ is non-trivial.
\end{theorem}

\begin{theorem}\label{thm'}
    For any odd prime $p$, there is a skew brace of order $2p^3$ whose automorphism group is trivial.
\end{theorem}

To prove Theorem \ref{thm'}, we shall introduce a method of constructing skew braces (Theorem \ref{thm:construction}) that is an extension of \cite[Theorem 3.2]{MM}.

\section{Preliminaries}

Let $A = (A,\cdot,\circ)$ be a skew brace. For each $a\in A$, define
\begin{align*}
\lambda_a : (A,\cdot) \rightarrow (A,\cdot);& \quad \lambda_a(b)= a^{-1}\cdot (a\circ b),\\
\rho_a : (A,\cdot)\rightarrow (A,\cdot); & \quad \rho_a(b) = (a\circ b)\cdot a^{-1},\end{align*}
which are clearly automorphisms of $(A,\cdot)$. It is well-known that
\begin{align*}
\lambda : (A,\circ)\rightarrow \Aut(A,\cdot);& \quad a\mapsto \lambda_a\\
\rho : (A,\circ)\rightarrow \Aut(A,\cdot);&\quad a\mapsto \rho_a
\end{align*}
are group homomorphisms. The two group operations are linked by 
\[ a\circ b = a\cdot \lambda_a(b),\quad a\cdot b =a\circ \lambda_{\overline{a}}(b),\quad \overline{a} = \lambda_{\overline{a}}(a^{-1})\]
\[ a\circ b = \rho_a(b)\cdot a,\quad b\cdot a = a\circ \rho_{\overline{a}}(b),\quad \overline{a}=\rho_{\overline{a}}(a^{-1}) \]
for all $a,b\in A$. Let us also recall the following basic definitions.

\begin{definition} A subset $I$ of $A$ is said to be:
\begin{enumerate}[(1)]
\item a \textit{sub-skew brace} if it is a subgroup of both $(A,\cdot)$ and $(A,\circ)$;
\item a \textit{left ideal} if it is a subgroup of $(A,\cdot)$ invariant under $\mathrm{Im}(\lambda)$;
\item an \textit{ideal} if it is a left ideal that is normal in both $(A,\cdot)$ and $(A,\circ)$.
\end{enumerate}
Note that if $I$ is a left ideal, then it is automatically a sub-skew brace, and we have $a\cdot I=a\circ I$ for all $a\in A$.
\end{definition}

Clearly $\ker(\lambda)$ and $\ker(\rho)$ are normal subgroups of $(A,\circ)$. Since
\[ a\circ b = \begin{cases}
a \cdot b & \mbox{for all }a\in \ker(\lambda),\, b\in A,\\
b \cdot a & \mbox{for all }a\in \ker(\rho),\, b\in A,
\end{cases}\]
we see that both of them are sub-skew braces. It is known that $\ker(\rho)$ is always a left ideal, but $\ker(\lambda)$ need not be a left ideal in general.

\begin{lem}\label{lem:invariant} The following hold:
\begin{enumerate}[$(a)$]
\item $\ker(\rho)$ is invariant under $\mathrm{Im}(\lambda)$;
\item $\ker(\lambda)$ is invariant under $\mathrm{Im}(\rho)$.
\end{enumerate}
\end{lem}
\begin{proof} We omit the proof of (a), which can be found in \cite[Proposition 2.2]{left simple}, because we only need (b) and the argument is very similar. 

For any $x\in \ker(\lambda)$ and $a\in A$, we can write
\begin{align*}
a\circ x\circ \overline{a} & = a\circ (x\cdot \overline{a})\\
& = \rho_a(x\cdot \overline{a})\cdot a\\
& = \rho_a(x)\cdot  a^{-1}\cdot a\\
& = \rho_a(x).
\end{align*}
Since $\ker(\lambda)$ is a normal subgroup of $(A,\circ)$, this implies that $\ker(\lambda)$ is invariant under the action of $\mathrm{Im}(\rho)$.
\end{proof}

The next lemma is helpful when studying $\Aut(A)$.

\begin{lem}\label{lem:circle hom} Let $f\in \Aut(A,\cdot)$. Then
\[ f\in \Aut(A,\circ) \iff \lambda_{f(a)} = f\lambda_af^{-1}\mbox{ for all }a\in A.\]
\end{lem}
\begin{proof} For any $a,b \in A$, we have
\[ f(a\circ b) = f(a\cdot \lambda_a(b)) ,\quad f(a)\circ f(b) =f(a) \cdot \lambda_{f(a)}(f(b)).\]
Since $f$ preserves the operation $\cdot$, we deduce that
\[ f(a\circ b) = f(a)\circ f(b) \iff (f\lambda_a)(b) = (\lambda_{f(a)}f)(b),\]
and the claim now follows.
\end{proof}

\section{Proof of Theorem \ref{thm}}

In what follows, let $A=(A,\cdot,\circ)$ be a skew brace with $|A|\geq 3$. We may assume that $\mathrm{Im}(\lambda)$ is non-trivial, for otherwise the two operations $\cdot$ and $\circ$ are equal, and $\Aut(A) = \Aut(A,\cdot)$ is non-trivial because $(A,\cdot)$ is a group of order at least $3$.

\subsection{Part (1)} For any $a\in A$, consider the inner automorphism
\[ \iota_a : (A,\circ) \rightarrow (A,\circ);\quad \iota_a(b) = \overline{a}\circ b \circ a\]
of $(A,\circ)$ induced by $a$. Since $A$ is two-sided, as is known in the literature, this map is also an automorphism of $(A,\cdot)$. Indeed, we have
\begin{align*}
\iota_a(b\cdot c) & = \overline{a}\circ (b\cdot c) \circ a\\
& = \left((\overline{a}\circ b)\cdot \overline{a}^{-1}\cdot (\overline{a}\circ c) \right) \circ a\\
& = (\overline{a}\circ b\circ a)\cdot a^{-1}\cdot (\overline{a}^{-1}\circ a) \cdot a^{-1}\cdot (\overline{a}\circ c\circ a)\\
& = \iota_a(b)\cdot \iota_a(c)
\end{align*}
for all $b,c\in A$. In the last equality, we used the fact that
\begin{align*}
a^{-1}\cdot (\overline{a}^{-1}\circ  a) \cdot a^{-1}
 & = a^{-1}\cdot (\overline{a}^{-1}\circ a)\cdot a^{-1}\cdot (\overline{a}\circ a)\\
 & = a^{-1}\cdot \left((\overline{a}^{-1}\cdot \overline{a}) \circ a\right)\\
  & = 1.
 \end{align*}
Since $(A,\circ)$ is non-abelian, we can find $a\notin Z(A,\circ)$ and the map $\iota_a$ is a non-trivial element of $\Aut(A)$.

\begin{remark} 
Instead of requiring $A$ to be two-sided, we only need the existence of an element $a\not\in Z(A,\circ)$ such that
\begin{equation}\label{eqn:right brace}
(b\cdot c)\circ a = (b\circ a)\cdot a^{-1}\cdot (c\circ a)
\end{equation}
holds for all $b,c\in A$, and the same proof shows that $\Aut(A)$ contains the non-trivial element $\iota_a$. In \cite{novel}, the elements $a\in A$ that satisfy \eqref{eqn:right brace} are called \textit{$u$-distributive}, and they give rise to set-theoretic solutions to the Yang-Baxter equation.\end{remark}

\subsection{Part (2)} 

Since we have \eqref{eqn:bi}, by exchanging the roles of the two operations $\cdot$ and $\circ$, for each $a\in A$, similarly the map
\[ \lambda_a' : (A,\circ) \rightarrow (A,\circ) ;\quad \lambda_a'(b) = \overline{a}\circ (a\cdot b)\]
is an automorphism of $(A,\circ)$. For any $a\in A$, observe that
\begin{align*}
\lambda_a'(b) & = \overline{a}\circ (a\cdot b)\\
& = \overline{a}\cdot \lambda_{\overline{a}}(a\cdot b)\\
& = \lambda_{\overline{a}}(a^{-1})\cdot \lambda_{\overline{a}}(a)\cdot \lambda_{\overline{a}}(b)\\
& = \lambda_a^{-1}(b)\end{align*}
for all $b\in A$. Thus, we have $\lambda_a' = \lambda_a^{-1}$, which in particular means that $\lambda_a$ is also an automorphism of $(A,\circ)$. We note that this fact is known by \cite{Caranti}. 
It now follows that
$\lambda_a\in \Aut(A)\mbox{ for all } a\in A$, 
and the claim now follows because $\mathrm{Im}(\lambda)$ is non-trivial.

\subsection{Part (3)}

The Sylow subgroups $P_1,\dots,P_d$ of $(A,\cdot)$ are ideals of $A$ because $(A,\cdot)$ and $(A,\circ)$ are both nilpotent. We then have
\begin{align*}
 A &= \prod_{i=1}^{d} (P_i,\cdot,\circ),\\
\Aut(A) &=  \prod_{i=1}^{d}\Aut(P_i,\cdot,\circ).
\end{align*}
Thus, it suffices to prove the claim when $|A|$ is a prime power, $p^n$ say.

Now, the group $(A,\circ)$ acts on $(A,\cdot)$ via the homomorphism $\rho$. Since $\ker(\lambda)$ is invariant under $\mathrm{Im}(\rho)$ by Lemma \ref{lem:invariant}, this induces an action
\begin{align*} (A,\circ)\times  \ker(\lambda)\backslash(A,\cdot)&\rightarrow\ker(\lambda)\backslash (A,\cdot)\\[6pt]
\quad (c,\ker(\lambda)a)&\mapsto \ker(\lambda)\rho_c(a)\end{align*}
of the $p$-group $(A,\circ)$ on the right coset space $ \ker(\lambda)\backslash(A,\cdot)$, whose size is a non-trivial power of $p$ because $\mathrm{Im}(\lambda)$ is non-trivial. Plainly $\ker(\lambda)$ is a fixed point, so the class equation gives another fixed point $\ker(\lambda)a$ with $a\not\in \ker(\lambda)$. For any $c\in A$, there then exists $z_c\in \ker(\lambda)$ such that $\rho_c(a)=z_ca$, which we rewrite as 
$\lambda_c(a) = c^{-1}z_cac$. We deduce that 
\[ \lambda_{\lambda_a(b)} = \lambda_a\lambda_b\lambda_a^{-1}\]
for all $b\in A$. Indeed, for $c = \lambda_a(b)$ and $z_c\in \ker(\lambda)$ as above, we have
\begin{align*}
\lambda_{\lambda_a(b)}\lambda_a& = \lambda_{c\circ a}\\
& = \lambda_{c\cdot \lambda_c(a)}\\
& = \lambda_{c\cdot c^{-1}z_cac}\\
& = \lambda_{z_c \circ ( a\lambda_a(b))}\\
& = \lambda_a\lambda_b.
\end{align*}
By Lemma \ref{lem:circle hom}, this shows that $\lambda_a$ is a non-trivial element of $\Aut(A)$.

\section{A method of constructing skew braces}

The next theorem is a slight generalization of \cite[Theorem 3.2]{MM}. For the convenience of the reader, let us recall its motivation.

Let $A = (A,\cdot,\circ)$ be any skew brace. Suppose that $A$ has an ideal $B$ and a sub-skew brace $C$ such that
\[ (A,\cdot) = (B,\cdot) \rtimes (C,\cdot),\quad (A,\circ) = (B,\circ) \rtimes (C,\circ).\]
These two semidirect product decompositions yield homomorphisms
\[ \phi : (C,\cdot)\rightarrow \Aut(B,\cdot),\quad \psi : (C,\circ) \rightarrow \Aut(B,\circ)\]
such that for any $b,b'\in B$ and $c,c'\in C$, we have
\begin{align*}
(b\cdot c)\cdot (b'\cdot c') & = (b\cdot \phi_c(b'))\cdot (c\cdot c'),\\
(c\circ b)\circ (c'\circ b') & = (c\circ c') \circ (\psi_{c'}^{-1}(b) \circ b').
\end{align*}
Since $B$ is a left ideal of $A$, we also have a homomorphism
\[ \gamma : (C,\circ) \rightarrow \Aut(B,\cdot);\quad c\mapsto \gamma_c:=\lambda_c|_B.\]
By identifying $b\cdot c$ with $(b,c)$ and therefore 
\[ c\circ b = c\cdot \lambda_c(b) = (\phi_c\gamma_c)(b)\cdot c\]
with $((\phi_c\gamma_c)(b),c)$, we then obtain the construction below. To simplify calculations, we assume that $(B,\cdot) = (B,+)$ is abelian, where we now use
$0$ to denote its identity, and the two operations coincide on $C$. We remark that \cite[Theorem 3.2]{MM} is the special case when $B = (B,+,+)$.
\begin{theorem}\label{thm:construction} Let $B = (B,+,\circ)$ be a skew brace with abelian additive group $(B,+)$, and let $C = (C,\cdot)$ be a group. For any homomorphisms
\[\phi,\gamma : C\rightarrow\Aut(B,+),\quad \psi : C \rightarrow \Aut(B,\circ)\]
satisfying the equalities
\begin{equation}\label{eqn:construction1}\gamma_c\lambda_{\psi_c^{-1}(b)} = \lambda_b\gamma_c, \quad \phi_c\lambda_b = \lambda_b\phi_c,\quad \phi_c\gamma_{c'}=\gamma_{c'}\phi_c\end{equation}
\begin{equation}\label{eqn:construction2} \phi_{c'}(\gamma_{cc'}\psi_{cc'}^{-1} - \gamma_{c'}\psi_{c'}^{-1}) = \gamma_c\psi^{-1}_c-\mathrm{id}_B\end{equation}
for all $b\in B$ and $c,c'\in C$, define 
\begin{align*}
    (b,c) \cdot (b',c') & = (b + \phi_c(b'),cc')\\
    ((\phi_c\gamma_c)(b),c)\circ ((\phi_{c'}\gamma_{c'})(b'),c')& = ( (\phi_{cc'}\gamma_{cc'})(\psi^{-1}_{c'}(b)\circ b'),cc')
\end{align*}
for any $b,b'\in B$ and $c,c'\in C$. Then $(B\times C,\cdot,\circ)$ is a skew brace.
\end{theorem}

\begin{proof}
    It is clear that $(B\times C,\cdot)$ is a group. Note that $\circ$ is nothing but the binary operation on $B\times C$ induced by the bijection
    \[ B\times C\rightarrow ((B,\circ^{\mbox{\tiny op}})\rtimes_\psi C)^{\mbox{\tiny op}}\quad (b,c)\mapsto ( (\phi_c\gamma_c)^{-1}(b),c^{-1}) \]
    via transport, where the superscript $^{\mbox{\tiny op}}$ denotes the opposite group, so it is also clear that $(B\times C,\circ)$ is a group. Note that
    \begin{align}\notag
    &( (\phi_c\gamma_c)(b),c)\circ ((\phi_{c'}\gamma_{c'})(b'),c')\\\notag
    &\hspace{1cm}= ( (\phi_{cc'}\gamma_{cc'})(\psi_{c'}^{-1}(b) + \lambda_{\psi_{c'}^{-1}(b)}(b')),cc')\\\label{eqn:circ}
    &\hspace{1cm}= ( (\phi_{cc'}\gamma_c)( (\gamma_{c'}\psi_{c'}^{-1})(b) + (\lambda_b\gamma_{c'})(b')),cc')
    \end{align}
for all $b,b'\in B$ and $c,c'\in C$ by the first equality in \eqref{eqn:construction1}.

Now, to prove the brace relation, let
    \[ a_i =( (\phi_{c_i}\gamma_{c_i})(b_i),c_i),\mbox{ where }b_i\in B,\, c_i\in C,\]
    for $i=1,2,3$. It is obvious that
    \begin{align*}
      \mbox{LHS}&:=a_1\circ (a_2\cdot a_3),\\
      \mbox{RHS}&:=(a_1\circ a_2)\cdot a_1^{-1}\cdot (a_1\circ a_3) \end{align*}
    both have $c_1c_2c_3$ in the second coordinate. Note that
    \begin{align*}
        a_2\cdot a_3 & = ( (\phi_{c_2}\gamma_{c_2})(b_2) + (\phi_{c_2c_3}\gamma_{c_3})(b_3),c_2c_3)\\
        & = ( (\phi_{c_2c_3}\gamma_{c_2c_3})( (\gamma_{c_3}^{-1}\phi_{c_3}^{-1})(b_2) + (\gamma_{c_2c_3}^{-1} \gamma_{c_3})(b_3)),c_2c_3)
    \end{align*}
    by the last equality in \eqref{eqn:construction1}, and so the first coordinate of LHS is 
    \begin{align*}
(\phi_{c_1c_2c_3}\gamma_{c_1})( (\gamma_{c_2c_3}\psi_{c_2c_3}^{-1})(b_1) + (\lambda_{b_1}\gamma_{c_2}\phi_{c_3}^{-1})(b_2) + (\lambda_{b_1}\gamma_{c_3})(b_3))
    \end{align*}
by \eqref{eqn:circ}. Similarly, again using \eqref{eqn:circ}, we compute that
 \begin{align*}
        a_1 \circ a_2 & = ( (\phi_{c_1c_2}\gamma_{c_1})(( \gamma_{c_2}\psi_{c_2}^{-1})(b_1) + (\lambda_{b_1}\gamma_{c_2})(b_2)),c_1c_2),\\
        a_1\circ a_3 & = ((\phi_{c_1c_3}\gamma_{c_1})( (\gamma_{c_3}\psi_{c_3}^{-1})(b_1) + (\lambda_{b_1}\gamma_{c_3})(b_3)),c_1c_3).
    \end{align*}
Since $a_1^{-1} = ( \gamma_{c_1}(-b_1),c_1^{-1})$,     we see that the first coordinate of RHS is 
    \begin{align*}
        &(\phi_{c_1c_2}\gamma_{c_1})( (\gamma_{c_2}\psi_{c_2}^{-1})(b_1) + (\lambda_{b_1}\gamma_{c_2})(b_2) - b_1) \\
        &\hspace{2cm}+ (\phi_{c_1c_2c_3}\gamma_{c_1})((\gamma_{c_3}\psi_{c_3}^{-1})(b_1) + (\lambda_{b_1}\gamma_{c_3})(b_3) ).
    \end{align*}
Using the middle and last equalities in \eqref{eqn:construction1}, and also the assumption that $(B,+)$ is abelian, the above simplifies to
       \begin{align*} 
       &(\phi_{c_1c_2c_3}\gamma_{c_1})( (\phi_{c_3}^{-1}(\gamma_{c_2}\psi_{c_2}^{-1}-\mathrm{id}_B) + \gamma_{c_3}\psi_{c_3}^{-1})(b_1)\\
       &\hspace{3.75cm} + (\lambda_{b_1}\gamma_{c_2}\phi_{c_3}^{-1})(b_2) + (\lambda_{b_1}\gamma_{c_3})(b_3)).
    \end{align*}
This equals the first coordinate of LHS by \eqref{eqn:construction2}, as desired.
\end{proof}

In the skew brace $(B\times C,\cdot,\circ)$  constructed  in  Theorem \ref{thm:construction},  we have that $B\times \{1\}$ and $\{0\}\times C$ are both sub-skew braces with
\[(b,1) \cdot (b',1) = (b+b',1),\quad (b,1)\circ (b',1) = (b\circ b',1)\]
\[(0,c)\cdot (0,c') = (0,cc'),\quad (0,c)\circ (0,c') = (0,cc')\]
for all $b,b'\in B$ and $c,c'\in C$. Clearly, in fact $B\times \{1\}$ is an ideal with 
\begin{align}\label{eqn:bc dot} (0,c) \cdot (\gamma_c(b),1) &=((\phi_c\gamma_c)(b),c)=  (( \phi_c\gamma_c)(b),1)\cdot (0,c)\\\label{eqn:bc circle} (0,c)\circ (b,1) &= ((\phi_c\gamma_c)(b),c)=( \psi_c(b),1) \circ (0,c)\end{align}
for all $b\in B$ and $c\in C$. In particular, we have
\begin{align*}
(B\times C,\cdot) &= (B \times \{1\},\cdot)\rtimes (\{0\}\times C),\\
(B\times C,\circ) &= (B \times \{1\},\circ)\rtimes (\{0\}\times C).
\end{align*}
Observe that for any $f\in\Aut(B\times C,\cdot,\circ)$, because it preserves both $\cdot$ and $\circ$,  there exists $\beta\in\Aut(B)$ such that
\[ f(b,1) = (\beta(b),1) \mbox{ for all }b\in B,\]
provided that $(B\times\{1\},\cdot)$ is a characteristic subgroup of $(B\times C,\cdot)$, or $(B\times\{1\},\circ)$ is a characteristic subgroup of $(B\times C,\circ)$.

\begin{prop}\label{prop:trivial Aut1}Suppose that we are in the setting of Theorem \ref{thm:construction}. Let $f\in \Aut(B\times C,\cdot,\circ)$ and $\beta\in \Aut(B)$ be such that
\[ f(b,1) = (\beta(b),1) \mbox{ for all }b\in B. \]
For any $c\in C$, write $f(0,c) = ( (\phi_{c_0}\gamma_{c_0})(b_0),c_0)$. Then
\[ \beta\phi_c= \phi_{c_0}\beta,\quad\beta \gamma_c = \gamma_{c_0}\lambda_{b_0}\beta,\mbox{ and } \]
\[(\lambda_{b_0}\beta -\psi_{c_0}^{-1}\beta\psi_c )(b)  =(\lambda_{ (\psi_{c_0}^{-1}\beta\psi_{c})(b)} - \mathrm{id}_B)(b_0) \]
has to hold for all $b\in B$. 
\end{prop}
\begin{proof}Let $b\in B$ and $c\in C$ be arbitrary.  Since $f$ preserves both of the operations $\cdot$ and $\circ$, we see from \eqref{eqn:bc dot} and \eqref{eqn:bc circle} that
\begin{align*}
f((\phi_c\gamma_c)(b),c)
&  = ((  \beta\phi_c\gamma_c)(b),1)\cdot  ((\phi_{c_0}\gamma_{c_0})(b_0),c_0)\\[2pt]
& = ((\phi_{c_0}\gamma_{c_0})(b_0),c_0) \cdot ((\beta\gamma_c)(b),1)\\[2pt]
& = ( (\beta \psi_c)(b),1) \circ ((\phi_{c_0}\gamma_{c_0})(b_0),c_0)\\[2pt]
& = ((\phi_{c_0}\gamma_{c_0})(b_0),c_0)\circ (\beta(b),1).
\end{align*}
It is straightforward to compute that then
\begin{align}\label{eqn1}
f((\phi_c\gamma_c)(b),c)
&  = ((  \beta\phi_c\gamma_c)(b)+(\phi_{c_0}\gamma_{c_0})(b_0),c_0)\\[2pt]\label{eqn2}
& = ((\phi_{c_0}\gamma_{c_0})(b_0) + (\phi_{c_0}\beta\gamma_c)(b),c_0)\\[2pt]\label{eqn3}
& = ( (\phi_{c_0}\gamma_{c_0})( (\psi_{c_0}^{-1}\beta\psi_c) (b) \circ b_0),c_0)\\[2pt]\label{eqn4}
& = ((\phi_{c_0}\gamma_{c_0})(b_0\circ \beta(b)),c_0).
\end{align}
By comparing \eqref{eqn1} and \eqref{eqn2}, we  immediately obtain the first desired equality. Similarly, by comparing \eqref{eqn2} and \eqref{eqn4}, we see that
\[ (\phi_{c_0}\gamma_{c_0})(b_0)+  (\phi_{c_0}\beta \gamma_c)(b) = (\phi_{c_0}\gamma_{c_0})(b_0 + (\lambda_{b_0}\beta)(b)),\]
which implies the second desired equality. Finally, by comparing \eqref{eqn3} and \eqref{eqn4}, we get that
\[ (\psi_{c_0}^{-1}\beta\psi_c)(b) + \lambda_{(\psi_{c_0}^{-1}\beta\psi_c)(b)}(b_0) = b_0 + (\lambda_{b_0}\beta)(b),\]
which clearly simplifies to the last desired equality.
 \end{proof}
 
\section{An explicit description of the skew brace of order $24$ with a trivial automorphism group}\label{sec:24}

 We shall apply Theorem \ref{thm:construction} to construct the unique skew brace
 \[ \texttt{SmallSkewbrace(24,855)}\]
of order $24$ that has a trivial automorphism group. We take
\[ B = (\mathbb{F}_2^2\times \mathbb{F}_3,+,\circ),\quad C = (\mathbb{F}_2,+),\]
where $+$ denotes the usual addition, and $\circ$ is defined by
\[  \left(\begin{pmatrix}
x_1\\x_2
\end{pmatrix}, z\right)\circ \left( \begin{pmatrix}y_1\\y_2\end{pmatrix},z'\right)= \left(\begin{pmatrix}
x_1\\x_2
\end{pmatrix}
+ \begin{pmatrix} 1 & 1 \\ 1 & 0 \end{pmatrix}^{z} \begin{pmatrix}y_1\\y_2\end{pmatrix},z+z'\right)\]
for any $(x_1,x_2)^T,(y_1,y_2)^T\in \mathbb{F}_2^2$ and $z,z'\in \mathbb{F}_3$. It is easy to verify that 
$B$ is a skew brace (see \cite[Example 1.4]{skew brace}). We shall identify
\[ \Aut(B,+) = \Aut(\mathbb{F}_2^2\times\mathbb{F}_3) = \mathrm{GL}_2(\mathbb{F}_2) \times \mathbb{F}_3^\times\]
in the obvious way. Note that then
\[ \lambda_{(\vec{x},z)} = \left( \begin{pmatrix} 1 &1 \\ 1 & 0 \end{pmatrix}^z,1\right)\]
for any $\vec{x}\in \mathbb{F}_2^2$ and $z\in \mathbb{F}_3$. Let us first compute $\Aut(B)$.

\begin{lem}\label{lem:Aut(B)0} The automorphism group of $B$ is given by
\begin{align*}
 \Aut(B) &  = \left\{ (M,1) : M \in \left\langle\left(\begin{smallmatrix}1 & 1 \\ 1 & 0 \end{smallmatrix}\right) \right\rangle\right\}\\
 & \hspace{2cm}\cup  \left\{ (M,-1) : M \not\in \left\langle\left(\begin{smallmatrix}1 & 1 \\ 1 & 0 \end{smallmatrix} \right)\right\rangle\right\}.\end{align*}
\end{lem}
\begin{proof}
It follows from Lemma \ref{lem:circle hom} that $\Aut(B)$ consists exactly of the pairs $(M,u)\in \Aut(B,+)$ such that 
\[  \lambda_{(M,u)(\vec{x},z)}   = (M,u) \lambda_{(\vec{x},z)} (M,u)^{-1}\]
for all $\vec{x}\in \mathbb{F}_2^2$ and $z\in\mathbb{F}_3$. We may rewrite the above as
\[ \left( \begin{pmatrix}1 & 1 \\ 1 & 0 \end{pmatrix}^{uz} ,1 \right) = \left( M\begin{pmatrix}1 & 1 \\ 1 & 0 \end{pmatrix}^{z}M^{-1},1\right).\]
Since $M\left( \begin{smallmatrix} 1 & 1 \\ 1 & 0 \end{smallmatrix}\right)M^{-1}=\left( \begin{smallmatrix} 1 & 1 \\ 1 & 0 \end{smallmatrix}\right)^{-1}$ when $M\not\in \left\langle\left( \begin{smallmatrix} 1 & 1 \\ 1 & 0 \end{smallmatrix}\right)\right\rangle$, we see that the above equality holds for all $z\in \mathbb{F}_3$ if and only if
\[ u =1\mbox{ and }M \in \langle \left( \begin{smallmatrix} 1 & 1 \\ 1 & 0 \end{smallmatrix}\right)\rangle,\quad \mbox{or}\quad 
\mbox{$u=-1$ and $M\not\in \langle \left( \begin{smallmatrix} 1 & 1 \\ 1 & 0 \end{smallmatrix}\right)\rangle$}.\]
This completes the proof.
\end{proof}

Now, consider the homomorphisms
\begin{align*}
    \phi : C\rightarrow \Aut(B,+);&\,\ c\mapsto \phi_c := \left(\left(\begin{smallmatrix}
    1 & 0 \\0 & 1
    \end{smallmatrix}\right),1\right),\\
    \gamma : C\rightarrow\Aut(B,+);&\,\ c\mapsto\gamma_c:= \left(\left(\begin{smallmatrix}
    0 & 1 \\ 1 & 0
    \end{smallmatrix}\right),-1\right)^c,\\
    \psi : C\rightarrow\Aut(B,\circ);&\,\ c\mapsto \psi_c:= \left(\left(\begin{smallmatrix}
    1 & 1 \\ 0 & 1
    \end{smallmatrix}\right),-1\right)^c.
\end{align*}
It is clear that $\gamma_1$ and $\psi_1$ have order $2$, and $\psi_1 \in \Aut(B,\circ)$ by Lemma \ref{lem:Aut(B)0}. We check that they satisfy the conditions in Theorem \ref{thm:construction}.

\begin{lem}\label{lem:conditions0}The homomorphisms $\phi,\gamma,\psi$ satisfy
\[\gamma_c\lambda_{\psi_c^{-1}(\vec{x},z)} = \lambda_{(\vec{x},z)}\gamma_c, \quad \phi_c\lambda_{(\vec{x},z)} = \lambda_{(\vec{x},z)}\phi_c,\quad \phi_c\gamma_{c'}=\gamma_{c'}\phi_c\]
\[\phi_{c'}(\gamma_{c+c'}\psi_{c+c'}^{-1} - \gamma_{c'}\psi_{c'}^{-1}) = \gamma_c\psi^{-1}_c-\mathrm{id}_B\]
for all $\vec{x}\in \mathbb{F}_2^2$, $z\in \mathbb{F}_3$, and $c,c'\in \mathbb{F}_2$. 
\end{lem}
\begin{proof}
All of the equalities are trivial when $c=0$ or $c'=0$. Hence, it suffices to verify them when $c=c'=1$. Since $\phi_1=\mathrm{id}_B$ by choice, the second and third equalities are trivial. Note that
\begin{align*}
\phi_1(\mathrm{id}_B - \gamma_1\psi_1^{-1}) & = \left(\left(\begin{smallmatrix}1 & 0 \\ 0 & 1  \end{smallmatrix} \right),1\right)-\left(\left(\begin{smallmatrix}0 & 1 \\ 1 & 1  \end{smallmatrix} \right) ,1\right)\\
&=\left(\left(\begin{smallmatrix}0 & 1 \\ 1 & 1  \end{smallmatrix} \right) ,1\right)-\left(\left(\begin{smallmatrix}1 & 0 \\ 0 & 1  \end{smallmatrix} \right),1\right)\\
& =\gamma_1\psi_1^{-1}-\mathrm{id}_B,
\end{align*}
so the last equality holds. Finally, we compute that
\begin{align*}
\gamma_1 \lambda_{\psi_1^{-1}(\vec{x},z)}\gamma_1^{-1} & = \left(\left(\begin{smallmatrix} 0 & 1 \\ 1 & 0\end{smallmatrix}\right), -1 \right)\left(\left(\begin{smallmatrix} 1& 1 \\ 1 & 0\end{smallmatrix}\right)^{-z}, 1 \right)\left(\left(\begin{smallmatrix} 0 & 1 \\ 1 & 0\end{smallmatrix}\right)^{-1}, -1 \right)\\
& = \left(\left(\begin{smallmatrix} 1 & 1 \\ 1 & 0\end{smallmatrix}\right)^{z}, 1 \right) \\
& = \lambda_{(\vec{x},z)},
\end{align*}
so the first equality is satisfied. This proves the claim.
\end{proof}

From Lemma \ref{lem:conditions0} and Theorem \ref{thm:construction}, we obtain a skew brace
\[ A= (B\times C,\cdot ,\circ) = ( (\mathbb{F}_2^2\times\mathbb{F}_3)\times \mathbb{F}_2,\cdot,\circ)\]
by defining the two operations by
\begin{align}\label{eqn:+def}
    ( (\vec{x},z),c) \cdot ((\vec{x}',z'),c') & =  ( (\vec{x}+\vec{x}', z+z') ,c+c')\\\notag
   (\gamma_c(\vec{x},z),c)\circ (\gamma_{c'}(\vec{x}',z'),c') & = (\gamma_{c+c'}( \psi_{c'}^{-1}(\vec{x},z) \circ (\vec{x}',z')) , c+ c')
\end{align}
for all $\vec{x},\vec{x}'\in \mathbb{F}_2^2$, $z,z'\in\mathbb{F}_3$, and $c,c'\in\mathbb{F}_2$. 

We prove that $\Aut(A)$ is trivial. Observe that
\begin{align*} (B,\circ) &= (\mathbb{F}_2^2,+) \rtimes (\mathbb{F}_3,+) \simeq A_4, \\
(B\times C,\circ) & = (B,\circ)\rtimes (\mathbb{F}_2,+) \simeq S_4,
\end{align*}
so we see that $(B\times \{0\},\circ)$ is a characteristic subgroup of $(B\times C,\circ)$. Thus, for any $f\in \Aut(A)$, there exists $\beta\in \Aut(B)$ such that
\[ f((\vec{x},z),0) = (\beta(\vec{x},z),0) \mbox{ for all }(\vec{x},z)\in \mathbb{F}_2^2\times \mathbb{F}_3.\]
By Lemma \ref{lem:Aut(B)0}, we may write
\[ \beta = (M,u),\mbox{ where } \begin{cases}
M \in \langle \left( \begin{smallmatrix}1 & 1 \\ 1 & 0 \end{smallmatrix}\right)\rangle &\mbox{for }u=1,\\
M \not\in \langle \left( \begin{smallmatrix}1 & 1 \\ 1 & 0 \end{smallmatrix}\right)\rangle &\mbox{for }u=-1.
\end{cases}\]
Since $f( (\vec{0},0),1)$ has order two in $(A,\cdot)$, it is clear from \eqref{eqn:+def} that
\[ f( (\vec{0},0),1) = (\gamma_1(\vec{y},0),1) = ( (\left(\begin{smallmatrix}y_2\\y_1\end{smallmatrix} \right),0),1)\]
for some $\vec{y}=(y_1,y_2)^T\in \mathbb{F}_2^2$. We need to show that $\beta = \mathrm{id}_B$ and $\vec{y}=\vec{0}$.

 Note that  $\lambda_{(\vec{y},0)} = \mathrm{id}_B$, so we deduce from Proposition \ref{prop:trivial Aut1}  that
\[\beta \gamma_1 = \gamma_{1}\beta, \mbox{ and }(\beta -\psi_{1}^{-1}\beta\psi_1 )(\vec{x},z)  =(\lambda_{ (\psi_{1}^{-1}\beta\psi_{1})(\vec{x},z)} - \mathrm{id}_B)( \vec{y},0) \]
has to hold for all $\vec{x} \in\mathbb{F}_2^2$ and $z\in \mathbb{F}_3$. The matrix $M$ has to commute with $\left(\begin{smallmatrix} 0 & 1 \\ 1 & 0 \end{smallmatrix}\right)$ by the first equality, and also with $\left(\begin{smallmatrix} 1 & 1 \\ 0 & 1 \end{smallmatrix}\right)$ by taking $z=0$ in the second equality. This yields $M = \left(\begin{smallmatrix} 1 & 0 \\ 0 & 1 \end{smallmatrix}\right)$ and $\beta = (\left(\begin{smallmatrix} 1 & 0 \\ 0 & 1 \end{smallmatrix}\right),1)=\mathrm{id}_B$. By taking $z=1$ in the second equality, we in turn see that
\[ \left( \left(\begin{smallmatrix} 1 & 1 \\ 1 & 0 \end{smallmatrix}\right) - \left(\begin{smallmatrix} 1 & 0 \\ 0 & 1 \end{smallmatrix}\right)\right)\vec{y} = \vec{0}\]
has to hold, so then $\vec{y}=\vec{0}$. We have thus shown that $f$ is the identity on both $B$ and $C$, whence $ f = \mathrm{id}_A$.
 
\section{Proof of Theorem \ref{thm'}}

Let $p$ be an odd prime. We shall apply Theorem \ref{thm:construction} to 
\[  B = (\mathbb{F}_p^3,+,\circ),\quad C = (\mathbb{F}_2,+),\]
where $+$ denotes the usual addition, and $\circ$ is defined by
\[ \begin{pmatrix}
    x_1 \\ x_2 \\ x_3 
\end{pmatrix}\circ \begin{pmatrix} y_1 \\ y_2\\ y_3\end{pmatrix} = \begin{pmatrix}
    x_1 + y_1 + x_3y_2\\ x_2 +y_2 \\ x_3 +y_3\end{pmatrix} \]
for any $(x_1,x_2,x_3)^T,(y_1,y_2,y_3)^T\in \mathbb{F}_p^3$ (see \cite[Theorem 3.2(3), Socle of order $p^2$]{bachiller}). We shall identify
\[ \Aut(B,+) = \Aut(\mathbb{F}_p^3) = \mathrm{GL}_3(\mathbb{F}_p)\]
in the obvious way. Note that then
\[ \lambda_{\vec{x}} = \begin{pmatrix} 1 & x_3 & 0 \\ 0 & 1 & 0 \\ 0 & 0 & 1 \end{pmatrix}\]
for any $\vec{x}=(x_1,x_2,x_3)^T\in \mathbb{F}_p^3$. Let us first compute $\Aut(B)$.

\begin{lem}\label{lem:Aut(B)} The automorphism group of $B$ is given by
\[ \Aut(B) = \left\{ \begin{pmatrix}
   uv & s & t \\
   0& u & 0\\
   0 & 0 & v
\end{pmatrix}: u,v\in \mathbb{F}_p^\times,\, s,t\in \mathbb{F}_p\right\}.\]
\end{lem}
\begin{proof} It follows from Lemma \ref{lem:circle hom} that $\Aut(B)$ consists exactly of the matrices $M\in \Aut(B,+)$ such that
\[ M\lambda_{\vec{x}} = \lambda_{M\vec{x}}M\]
 for all $\vec{x}\in \mathbb{F}_p^3$. Put $ \vec{x} = (x_1,x_2,x_3)^T$ and $M\vec{x} = (x_1',x_2',x_3')^T$. Writing $M=(m_{ij})$ for the entries of $M$, we compute that
\begin{align*}
M\lambda_{\vec{x}} &= \begin{pmatrix}
    m_{11} & m_{12}+m_{11}x_3 & m_{13}\\
    m_{21} & m_{22}+m_{21}x_3 & m_{23}\\
    m_{31} & m_{32}+m_{31}x_3 & m_{33}
\end{pmatrix},\\[6pt]
\lambda_{M\vec{x}}M &= \begin{pmatrix}
    m_{11} + m_{21}x_3' & m_{12} + m_{22}x_3' & m_{13}+m_{23}x_3'\\
    m_{21}&m_{22}&m_{23}\\
    m_{31} &m_{32}&m_{33}
\end{pmatrix}.
\end{align*}
For them to be equal for all $\vec{x}\in\mathbb{F}_p^3$, we need $m_{21}=m_{23}=m_{31}=0$, in which case $x_3' = m_{32}x_2+m_{33}x_3$ and we have
\begin{align*}
M\lambda_{\vec{x}} &= \begin{pmatrix}
    m_{11} & m_{12}+m_{11}x_3 & m_{13}\\
    0 & m_{22} & 0\\
    0 & m_{32} & m_{33}
\end{pmatrix},\\
\lambda_{M\vec{x}}M &= \begin{pmatrix}
    m_{11} & m_{12} + m_{22}(m_{32}x_2 + m_{33}x_3) & m_{13}\\
    0&m_{22}&0\\
    0 &m_{32}&m_{33}
\end{pmatrix}.
\end{align*}
Since $m_{22}\neq 0$ by the invertibility of $M$, we now deduce that
\begin{align*}
& M \lambda_{\vec{x}} = \lambda_{M\vec{x}}M \mbox{ for all }\vec{x}\in \mathbb{F}_p^3 \\
& \hspace{1cm}\iff m_{21}=m_{23}=m_{31}=m_{32} =0,\, m_{11}=m_{22}m_{33}.\end{align*}
This completes the proof.
\end{proof}

Now, let us fix the parameters
\[\epsilon \in \{-1,1\}, \quad \delta_2\in\mathbb{F}_p^\times,\quad \delta_1,\delta_3,\delta_4\in\mathbb{F}_p\]
satisfying the non-equality
\begin{equation}\label{eqn:delta}4(\delta_3-\epsilon \delta_4) + \delta_1\delta_2(1-\epsilon)\neq 0.\end{equation}
Consider the homomorphisms
\begin{align*}
    \phi : C\rightarrow \Aut(B,+);&\,\ c\mapsto \phi_c := \left(\begin{smallmatrix}
        -1 & \frac{1}{2}\delta_1\delta_2 & \delta_1 \\
        0 & -1 & 0 \\
        0 & \delta_2 & 1 
    \end{smallmatrix}\right)^c,\\[6pt]
    \gamma : C\rightarrow\Aut(B,+);&\,\ c\mapsto\gamma_c:=\left(\begin{smallmatrix}
        \epsilon & \delta_3 & -\frac{1}{2}(\epsilon+1)\delta_1\\
        0 & -\epsilon & 0\\
        0 & \frac{1}{2}(\epsilon-1)\delta_2 & -1
            \end{smallmatrix}\right)^c,\\[6pt]
    \psi : C\rightarrow\Aut(B,\circ);&\,\ c\mapsto \psi_c:=\left(\begin{smallmatrix}
        1 & \delta_4 & -\frac{1}{2}\delta_1 \\
        0 & -1 & 0 \\
        0 & 0 & -1
    \end{smallmatrix}\right)^c.
\end{align*}
It is straightforward to check that $\phi_1,\gamma_1$, and $\psi_1$ all have order $2$, and note that $\psi_1 \in \Aut(B,\circ)$ by Lemma \ref{lem:Aut(B)}. Let us show that they indeed satisfy the conditions in Theorem \ref{thm:construction}.

\begin{lem}\label{lem:conditions}The homomorphisms $\phi,\gamma,\psi$ satisfy
\[\gamma_c\lambda_{\psi_c^{-1}(\vec{x})} = \lambda_{\vec{x}}\gamma_c, \quad \phi_c\lambda_{\vec{x}} = \lambda_{\vec{x}}\phi_c,\quad \phi_c\gamma_{c'}=\gamma_{c'}\phi_c\]
\[\phi_{c'}(\gamma_{c+c'}\psi_{c+c'}^{-1} - \gamma_{c'}\psi_{c'}^{-1}) = \gamma_c\psi^{-1}_c-\mathrm{id}_B\]
for all $\vec{x}\in \mathbb{F}_p^3$ and $c,c'\in \mathbb{F}_2$. 
\end{lem}
\begin{proof}All of the equalities are trivial when $c=0$ or $c'=0$. Hence, it is enough to verify them when $c=c'=1$.
    Write $\vec{x}=(x_1,x_2,x_3)^T$, and note that $\psi_{1}^{-1}(\vec{x}) = (*,*,-x_3)^T$. It is straightforward to check that
\begin{align*}
\gamma_1\lambda_{\psi_1^{-1}(\vec{x})} & = 
\begin{pmatrix}
\epsilon & \delta_3 - \epsilon x_3 & -\frac{1}{2}\delta_1(\epsilon+1)\\
0 & -\epsilon & 0 \\
0 & \frac{1}{2}\delta_2(\epsilon-1) & -1
\end{pmatrix}
 =\lambda_{\vec{x}}\gamma_1,\\
\phi_1\lambda_{\vec{x}} & = 
\begin{pmatrix}
-1 & -x_3 + \frac{1}{2}\delta_1\delta_2 & \delta_1\\
0 & -1 & 0 \\
0 & \delta_2 & 1
\end{pmatrix}
=\lambda_{\vec{x}}\phi_1,\\
\phi_1\gamma_1 & =
\begin{pmatrix}
-\epsilon & -\delta_3 -\frac{1}{2}\delta_1\delta_2 & \frac{1}{2}\delta_1(\epsilon-1)\\
0 & \epsilon & 0 \\
0 & -\frac{1}{2}\delta_2(\epsilon+1) & - 1
\end{pmatrix}=\gamma_1\phi_1,
\end{align*}
and this proves the first three equalities. Similarly, we compute that
\begin{align*}
\phi_1(\mathrm{id}_B - \gamma_1\psi_1^{-1}) & =
\begin{pmatrix}
-1 & \frac{1}{2}\delta_1\delta_2 & \delta_1\\
0 & -1 & 0 \\
0 & \delta_2 & 1
\end{pmatrix}
\begin{pmatrix}
1-\epsilon & \delta_3 -\delta_4\epsilon & -\frac{1}{2}\delta_1\\
0 & 1-\epsilon & 0\\
0 & \frac{1}{2}\delta_2(\epsilon-1) & 0
\end{pmatrix} \\
& =
\begin{pmatrix}
\epsilon-1 & \delta_4\epsilon-\delta_3   & \frac{1}{2}\delta_1\\
0 & \epsilon-1 & 0\\
0 & \frac{1}{2}\delta_2(1-\epsilon) & 0
\end{pmatrix} = \gamma_1\psi_1^{-1}-\mathrm{id}_B,
\end{align*}
and this proves the last equality.
\end{proof}

From Lemma \ref{lem:conditions} and Theorem \ref{thm:construction}, we obtain a skew brace
\[ A= (B\times C,\cdot ,\circ) = (\mathbb{F}_p^3\times \mathbb{F}_2,\cdot,\circ)\]
by defining the two operations by
\begin{align*}
    (\vec{x},c) \cdot (\vec{x}',c') & =  (\vec{x}+\phi_c(\vec{x}') ,c+c')\\
    ( (\phi_c\gamma_c)(\vec{x}),c)\circ ((\phi_{c'}\gamma_{c'})(\vec{x}'),c') & = ((\phi_{c+c'}\gamma_{c+c'})( \psi_{c'}^{-1}(\vec{x}) \circ \vec{x}'), c+ c')
\end{align*}
for all $\vec{x},\vec{x}'\in \mathbb{F}_p^3$ and $c,c'\in\mathbb{F}_2$. 

\begin{lem}\label{lem:order two}
    Let  $\vec{y} = (y_1,y_2,y_3)^T\in \mathbb{F}_p^3$ be such that 
    \begin{align*}
    ((\phi_1\gamma_1)(\vec{y}),1) \cdot ((\phi_1\gamma_1)(\vec{y}),1)& =(\vec{0},0),\\
    ((\phi_1\gamma_1)(\vec{y}),1) \circ ((\phi_1\gamma_1)(\vec{y}),1) & =(\vec{0},0).
    \end{align*}
    Then $\delta_2y_2 + 2y_3=0$ and $2y_1 + \delta_4y_2 - \frac{1}{2}\delta_1y_3-y_3y_2=0$.
\end{lem}
\begin{proof}
    Observe that
      \begin{align*}
     ( (\phi_1\gamma_1)(\vec{y}),1) \cdot (  ( \phi_1\gamma_1)(\vec{y}),1) 
     & = ( (\phi_1\gamma_1)(\vec{y}) + (\phi_1\phi_1\gamma_1)(\vec{y}),0)\\
     & =  \left( \left( \begin{smallmatrix} 0 &  -\frac{1}{2} \delta_1\delta_2 & -\delta_1 \\0 & 0 & 0 \\ 0 & -\delta_2 & -2 \end{smallmatrix}\right)\left(\begin{smallmatrix}y_1\\y_2\\y_3\end{smallmatrix}\right), 0 \right)\\
     &= \left(\left( \begin{smallmatrix} -\frac{1}{2}\delta_1\delta_2y_2 - \delta_1y_3 \\ 0 \\ -\delta_2y_2 -2y_3\end{smallmatrix}\right), 0 \right),
    \end{align*}    
    and similarly we have
    \begin{align*}
    ( (\phi_1\gamma_1)(\vec{y}),1) \circ (  ( \phi_1\gamma_1)(\vec{y}),1) 
     & = ( (\phi_0\gamma_0)( \psi_1^{-1}(\vec{y}) \circ \vec{y}), 0)\\
     & = \left( \left(\begin{smallmatrix} 1 & \delta_4 & 
     -\frac{1}{2}\delta_1\\ 0 & -1 &0 \\0 & 0 & -1\end{smallmatrix}\right)\left(\begin{smallmatrix}y_1\\y_2\\y_3\end{smallmatrix}\right)\circ \left(\begin{smallmatrix}y_1\\y_2\\y_3\end{smallmatrix}\right) ,0\right)\\
     & = \left( \left(\begin{smallmatrix} 2y_1 + \delta_4y_2-\frac{1}{2}\delta_1y_3 -y_3y_2\\0\\0\end{smallmatrix}\right) ,0\right).
    \end{align*}
The claim is now clear because the above are equal to $(\vec{0},0)$.
\end{proof}

We are now ready to prove Theorem \ref{thm'} by showing that $\Aut(A)$ is trivial. Note that $\mathbb{F}_p^3\times \{0\}$ is a normal Sylow $p$-subgroup and hence a characteristic subgroup of $(A,\cdot)$. This means that for any $f\in \Aut(A)$, there exists $\beta\in \Aut(B)$ such that 
\[ f(\vec{x},0) = (\beta(\vec{x}),0) \mbox{ for all }\vec{x}\in \mathbb{F}_p^3.\]
By Lemma \ref{lem:Aut(B)}, we may write
\[ \beta = \begin{pmatrix}
   uv & s & t \\
   0& u & 0\\
   0 & 0 & v
\end{pmatrix},\mbox{ where }u,v\in \mathbb{F}_p^\times,\, s,t\in \mathbb{F}_p.\]
Since $f(\vec{0},1)$ has order two in both of  $(A,\cdot)$ and $(A,\circ)$, we can write 
\[ f(\vec{0},1) = ( (\phi_1\gamma_1)(\vec{y}),1),\]
where $\vec{y}=(y_1,y_2,y_3)^T\in \mathbb{F}_p^3$ has to satisfy
\begin{equation}\label{eqn:y} 
\begin{cases}
\delta_2y_2 + 2y_3=0\\[6pt]
2y_1 + \delta_4y_2 - \frac{1}{2}\delta_1y_3-y_3y_2=0
\end{cases}\end{equation}
by Lemma \ref{lem:order two}. We need to show that $\beta = \mathrm{id}_B$ and $\vec{y}=\vec{0}$.

From Proposition \ref{prop:trivial Aut1}, we know that
\[ \beta\phi_1= \phi_{1}\beta,\quad\beta \gamma_1 = \gamma_{1}\lambda_{\vec{y}}\beta,\mbox{ and } \]
\begin{equation}\label{eqn:lambda}(\lambda_{\vec{y}}\beta -\psi_{1}^{-1}\beta\psi_1 )(\vec{x})  =(\lambda_{ (\psi_{1}^{-1}\beta\psi_{1})(\vec{x})} - \mathrm{id}_B)(\vec{y}) \end{equation}
has to hold for all $\vec{x} = (x_1,x_2,x_3)^T\in\mathbb{F}_p^3$. By comparing the $(1,3)$- and $(3,2)$-entries in  $\beta\phi_1=\phi_1\beta$, we get that
\[ t + \delta_1uv = -t +\delta_1v,\quad \delta_2v =\delta_2u.\]
Since $\delta_2\neq 0$ by choice, they imply that
\[ u =v,\quad t = \frac{1}{2}\delta_1 u(1-u).\]
From here, we conduct our calculations with 
\[ \beta = \begin{pmatrix}
u^2 & s & \frac{1}{2}\delta_1 u(1-u)\\
0 & u & 0\\
0 & 0 & u
\end{pmatrix}.\]
By comparing the $(1,2)$-entries in $\beta\gamma_1=\gamma_1\lambda_{\vec{y}}\beta$, we see that
\[-\epsilon s + \delta_3u^2 + \frac{1}{4}\delta_1\delta_2u(u-1)(1-\epsilon) = \epsilon s + \delta_3u + \epsilon u y_3,
\]
which simplifies to
\begin{equation}\label{eqn:s1}
2\epsilon s + \epsilon uy_3 + u(1-u)\left( \delta_3 + \frac{1}{4}\delta_1\delta_2(1-\epsilon)\right)=0.
 \end{equation}
Moreover, we compute that
\[ \lambda_{\vec{y}}\beta - \psi_1^{-1}\beta\psi_1 = \begin{pmatrix}
0 & 2s +\delta_4u(1-u) + uy_3 & \frac{1}{2}\delta_1u(1-u)\\
0 & 0 & 0 \\
0 & 0 & 0 
\end{pmatrix}.\]
Note that $(\psi_1^{-1}\beta \psi_1)(\vec{x}) = (*,*,ux_3)$. Thus, by taking $(x_2,x_3)=(1,0)$, in which case $\lambda_{(\psi_1^{-1}\beta \psi_1)(\vec{x})} = \mathrm{id}_B$, we deduce from \eqref{eqn:lambda} that
\begin{equation}\label{eqn:s2}
 2s  + \delta_4u(1-u) + uy_3 =0\end{equation}
has to hold. By taking $x_3=1$, similarly we see from \eqref{eqn:lambda} that
\begin{equation}\label{eqn:y2} \frac{1}{2}\delta_1u(1-u) = u y_2. \end{equation}
Now, multiplying \eqref{eqn:s2} by $-\epsilon$ and then adding it to \eqref{eqn:s1}, we obtain
\[ u(1-u) \left(\delta_3 + \frac{1}{4}\delta_1\delta_2(1-\epsilon) -\epsilon \delta_4 \right)=0.\]
The term in the large parentheses is non-zero by our assumption \eqref{eqn:delta}, so it follows that $u=1$. We then obtain from \eqref{eqn:y2} that $y_2=0$, and in turn from \eqref{eqn:y} that $y_3=y_1=0$. This yields $\vec{y}=\vec{0}$. Note that $s =0$ has to hold now by \eqref{eqn:s2} and so $\beta = \mathrm{id}_B$. We have thus shown that $f$ is the identity on both $B$ and $C$, whence $ f = \mathrm{id}_A$.

\section*{Acknowledgments}

This work is supported by JSPS KAKENHI 24K16891.

\end{document}